\documentclass[11pt]{amsart}
\usepackage{pstricks}
\usepackage{pst-node}
\usepackage{amscd}
\usepackage{amssymb}
\usepackage[all]{xy}

\newtheorem{Thm}{Theorem}[subsection]
\newtheorem{Conj}[Thm]{Conjecture}
\newtheorem{Prop}[Thm]{Proposition}
\newtheorem{Def}[Thm]{Definition}
\newtheorem{Def/Thm}[Thm]{Definition/Theorem}
\newtheorem{Cor}[Thm]{Corollary}
\newtheorem{Lemma}[Thm]{Lemma}
\theoremstyle{remark}
\newtheorem{Rmk}[Thm]{Remark}
\newtheorem{EG}[Thm]{Example}
\newcommand{\ti}{\times}

\newcommand{\ra}{\rightarrow}

\DeclareMathOperator{\Spec}{Spec}

\newcommand{\X}{{\mathfrak X}}
\newcommand{\cW}{{\mathcal W}}
\newcommand{\cC}{{\mathcal C}}
\newcommand{\bk}{{\mathbf k}}

\newcommand{\A}{{\mathbb A}}
\newcommand{\PP}{{\mathbb P}}

\newcommand{\ZZ}{{\mathbb Z}}

\begin{document}

\title{A compactification of the space of maps from curves}

\author{Bumsig Kim}
\address{School of Mathematics, Korea Institute for Advanced Study,
85 Heogiro, Dongdaemun-gu, Seoul, 130-722, Rep.\ of Korea}
\email{bumsig@kias.re.kr}

\author{Andrew Kresch}
\address{Institut f\"ur Mathematik,
Universit\"at Z\"urich, Winterthurerstrasse 190, CH-8057 Z\"urich,
Switzerland}
\email{andrew.kresch@math.uzh.ch}

\author{Yong-Geun Oh}
\address{Mathematics Department,
480 Lincoln Dr., Madison WI 53706-1388, USA, and \\
Department of Mathematics, POSTECH,
Pohang, 790-784, Rep.\ of Korea}
\email{oh@math.wisc.edu}

\begin{abstract}
We construct a new compactification of the moduli space of maps
from pointed nonsingular projective stable curves to a
nonsingular  projective variety with prescribed
ramification indices at the points. It is shown to be a proper
Deligne-Mumford stack equipped with a natural virtual fundamental class.
\end{abstract}


\maketitle


\section{Introduction}
\subsection{Overview and motivation}
Let $\mu = (\mu _1,\ldots,\mu _n)$, $\mu _i\in \ZZ _{\ge 1}$. In this
paper, we construct a new compactification
$\overline{\mathfrak{U}}_{g, \mu }(X,\beta )$ of the space of all
maps $f$ from genus $g$, $n$-pointed nonsingular
projective curves $(C,p_1,\ldots,p_n)$ to a nonsingular
projective variety $X$,
representing class $\beta \in A_1(X)/\sim ^{\mathrm{alg}}$
such that:

\begin{itemize}

\item $f$ is unramified everywhere except possibly at $p_1$, $\ldots$, $p_n$,
where the ramification indices of $f$ are $\mu_1$, $\ldots$, $\mu _n$, respectively.

\item $f(p_i)$, $i=1$, $\ldots$, $n$, are pairwise distinct.

\end{itemize}
Here we follow the convention that the ramification index of $f$ at a point $p\in C$ is $1$
if $f$ is unramified at $p$.

The boundary of $  \overline{\mathfrak{U}}_{g, \mu }(X,\beta )   $    
consists of suitable maps, still keeping the
conditions in a sense, from $n$-pointed prestable genus $g$ curves
to Fulton-MacPherson degeneration spaces of $X$.  A
Fulton-MacPherson degeneration space  is, by definition, a fiber of the
``universal'' family $X[m]^+\ra X[m]$ of the Fulton-MacPherson
configuration space $X[m]$ of $m$ distinct labeled points in $X$, for
some integer $m$ (Definition \ref{FMdeg}). The elements in the
moduli space are called stable ramified maps (Definition
\ref{DefPmap}). They can be considered as stable log-unramified
maps. Stable ramified maps are always finite maps.
We show that the moduli space
$\overline{\mathfrak{U}}_{g, \mu }(X,\beta )$ is a proper
Deligne-Mumford stack over an algebraically closed base field $\bk$ of
characteristic zero, carrying a natural virtual
fundamental class (Theorem \ref{MainThm} and Section
\ref{Deformation}). Therefore, we will be able to define ramified Gromov-Witten
invariants (Section \ref{Deformation}). 
This compactification   
incorporates four known spaces in
algebraic geometry: Fulton-MacPherson's configuration space,
Kontsevich's stable map compactification, Harris-Mumford's admissible cover compactification,
and Li's compactification of the moduli of stable relative maps. All of
the four will play key roles in our construction.

This paper is motivated by the third author's
discussions with Fukaya on a similar compactification in the
context of  almost K\"ahler (or symplectic) geometry.
A further motivation is a conjectural link with BPS counts
due to Pandharipande (Section \ref{rahul}).

The main advantage of the algebro-geometric method exploited in this
paper is the systematic use of the wonderful property of the
universal families of Fulton-MacPherson spaces \cite{FM,
L.Li}, the notion of (log) admissible maps \cite{HM, J.Li, Mo, We},
the deformation theory for such maps \cite{J.Li2, Olsson3}, as well
as the now-standard tools in Gromov-Witten theory in
the algebro-geometric category \cite{BF, FP, KM, LT}.
In \cite{Kim}, a variant of this new compactification
is shown to be a smooth irreducible proper Deligne-Mumford stack,
compactifying  the space of maps from elliptic curves to a projective space.

\subsection{Conventions}
Let $\bk$ be a base field, which is algebraically closed and of
characteristic zero.
Every space will be a Noetherian $\bk$-scheme unless otherwise specified.
We often use $W|_T$ or $h^*W$ to denote the fiber product $W\ti _ST$ of algebraic spaces,
where $h:T\ra S$ is the map in the fiber square.
Expect where otherwise mentioned, $R$ will denote the DVR
$\bk[[t]]$, with field of fractions $K=k((t))$.
The extension over $R$ of an object over $K$ will, without explicit mention,
entail the passage to a finite extension of $K$, by which $R$ gets replaced
by its integral closure in the extension field.

\subsection{Acknowledgements} The authors thank
D.~Abramovich, D.~Cheong, I.~Ciocan-Fontanine, 
B.~Fantechi, K.~Fukaya, L.~Li, R.~Pandharipande,
F.~Sato, R.~Vakil, and D.~Zuo
for stimulating discussions.
Additionally,
R.~Pandharipande is thanked for providing the content of
Section \ref{rahul}, which is entirely due to him.
B.K. is partially supported by the NRF grant 2011-0001181 through ASARC.
A.K. is partially supported by a grant from the SNF\@. Y.O. is partially
supported by the NSF grant DMS 0904197 and acknowledges the financial support
and excellent research environment of KIAS over many years of summer visits.

\section{Stacks of Fulton-MacPherson degeneration
spaces}\label{X[n]}
\subsection{FM degeneration spaces} Let $X$ be a nonsingular  variety
of dimension $r\ge 1$. Denote by $X[n]$ the
Fulton-MacPherson configuration space of $n$ distinct labeled
points in $X$. We refer to the paper \cite{FM} for the constructions
and the basic properties of the configuration space. The space $X[n]$
has the ``universal family''
\[ \pi _{X[n]}: X[n]^+\ra X[n] \] with disjoint sections
\[ \sigma _i : X[n]\ra X[n]^+ \] for $i=1$, $\ldots$, $n$.
The universal space $X[n]^+$ is an iterated blowup of $X[n]\times
X$ along smooth centers.
Hence, there are two natural projections: $\pi _{X[n]}$ as
mentioned above, and
\[ \pi _X: X[n]^+\ra X \] from the second projection.

\begin{Def}\label{FMdeg} A pair $$ (\pi _{\cW /S}:\cW \ra S, \, \pi _{\cW /X} :\cW  \ra
X)$$ of morphisms is called a {\em Fulton-MacPherson  degeneration space of $X$
over a scheme $S$} (for short, an {\em FM
space of $X$ over $S$}) if:
\begin{itemize}

\item $\cW$ is an algebraic space.

\item There exists an \'etale surjective map $T\ra S$ from a scheme $T$,
 an integer $n>0$, and a fiber square
         \[
         \xymatrix{ \cW |_T \ar[r] \ar[d] & X[n]^+ \ar[d]^{\pi_{X[n]}} \\
         T \ar[r] & X[n]}
         \]
         such  that the pullback  of $\pi _{\cW /X}$
            to $\cW |_{T}$ coincides with the composite
            $\cW |_{T}\ra X[n]^+\ra X$.

\end{itemize}

Furthermore, when $S$ is $\Spec \bk$, we will simply say that
$\cW$ is an {\em FM  space of $X$.} When $n$ is
specified, we will call $\cW$ a {\em level-$n$ FM space.}
\end{Def}

\begin{EG}\label{speed} Let $S=\Spec R = \Spec \bk [[t]]$, and let $g$
be a morphism from $S$ to $X\times X$. Suppose that only the closed point $p\in S$ hits
the diagonal $\Delta$ of $X\times X$ under the map $g$; then there is
a unique lift $\tilde{g}: S \ra X[2]=\mathrm{Bl}_{\Delta}X\times X$.
Let $k$ be the intersection number of $S$ and the exceptional
divisor of $X[2]$, and let $q=g(p)$. By definition,
$X[2]^+$ is the blowup
of $X[2]\times X$ along the proper transform of the small diagonal
$\Delta _{\{ 1,2,3\}}$ of $X\times X\times X$.  We see, by direct computation, that
$\tilde{g}^*X[2]^+$ is isomorphic, as an FM  space
over $S$, to the pullback of the blowup $\mathrm{Bl}_{(p, q)}
S\times X$ of $S\times X$ at the point $(p,q)$, under the base change
\begin{align*} S &\ra S \\
             t&\mapsto  at^k+o(t^k)
\end{align*}
for some nonzero $a\in \bk$.
\end{EG}

One of the goals of Section \ref{X[n]} is to
construct the Artin stack which parameterizes FM
spaces of $X$. To do so, we will recall some elementary properties
of  $X[n]$, $X[n]^+$, and $\pi _{X[n]}$. Using a variant of the FM configuration spaces,
we will see that the stack is described by means of a smooth groupoid
scheme, and it will follow that the stack is algebraic.

\subsection{Basic properties of $\pi _{X[n]} :X[n]^+\ra X[n]$}
In \cite{FM}, the space $X[n]$ is constructed by an iterated
blowup of $X^n$ along nonsingular subvarieties $\Delta _I$ for all
subsets $I\subset N:=\{ 1,2,\ldots,n\}$ with $|I|\ge 2$, where $\Delta
_I$ is the proper transform of the diagonal
\[\{ (x_1,\ldots,x_n) \in X^n \ | \ x_i = x_j, \ \forall\ i, j\in I\}. \] The blowup
order of centers is suitably taken. For convenience, we use the following notation.

\medskip

\noindent{\em Notation}: Let $V_1$ be the blowup of a nonsingular
variety $V_0$ along a nonsingular closed subvariety $Z$. If $Y$ is
an irreducible subvariety of $V_0$, we will use the same notation
$Y$ to denote

\begin{itemize}
\item the total transform of $Y$, if $Y \subset Z$;

\item the proper transform of $Y$, otherwise.
\end{itemize}

This abuse of notation causes no confusion as long as it is
clear to which space $Y$ belongs.

\medskip

The  universal family $X[n]^+$ is  an
iterated blowup \[ X[n]^+:=Y_{n-1} \ra Y_{n-2} \ra \cdots\ra
Y_0:=X[n]\times X .\] The intermediate space $Y_{k+1}$ is the
blowup of $Y_{k}$ along all disjoint nonsingular subvarieties
$\Delta _{I^+}$ for $I\subset N$ and $|I|= n-k$, where $I^+
:= I\cup \{ n+1\}$.

\subsection{A local description of $\pi_{X[n]}$}\label{local-description}
Notice that in any stage $Y_k$, $\Delta _{I^+}$ is the transverse
intersection $\Delta _I \cap \Delta _{a^+}$ if $I\subset N$,
$2\le |I|\le n-k$ and $a\in I$, where $a^+:=\{a,n+1\}$. Also
note that in each $Y_k$, the intersection of $\Delta _{a^+}$ with every
fiber $F_k$ of the natural projection $Y_k\ra X[n]$ is transverse. These observations
provide a local description of the projections $\pi _{k+1}:
Y_{k+1}\ra X[n]$ as follows. If $p$ is a singular point of the
projection $\pi _{k+1}$, then with respect to suitable coordinates the
induced map on completed local rings has the form
\begin{align*}
\widehat{\mathcal{O}}_{\pi _{X[n]} (p)}\cong \bk
[[t_1,\ldots,t_{rn}]]
&\ra
\widehat{\mathcal{O}}_p\cong \bk
[[t_1,\ldots,t_{rn},z_1,\ldots,z_{r+1}]]/(z_1z_2-t_1) \\
t_i &\mapsto t_i,\end{align*}
and the completed local ring of the fiber $F_k$ at the point corresponding to $p$ can be identified with
$\bk[[z_1, z_1z_3, \ldots, z_1z_{r+1}]]$.


\medskip
\noindent{\em Projective Tangent Bundle Map.} 
Since the general fiber of $\pi_{X[n]}$ is just $X$ there is an induced
rational map from $\PP(T (X[n]^+/X[n]))|_{(X[n]^+)^{\mathrm{sm}}}$,
the projectivization of the relative tangent sheaf restricted
to the locus of smooth points of $\pi_{X[n]}$, to $\PP(TX)$.
From the local description it follows that this is actually a morphism
\[ \PP(T (X[n]^+/X[n])|_{(X[n]^+)^{\mathrm{sm}}})\ra \PP(TX).\]
On the level of fibers it says that if $W$ is an FM degeneration space and
$p$ is a smooth point of $W$, then
$\PP T_pW$ is identified naturally
with $\PP T_{\pi _{W/X}(p)}X$.
This identification is independent of the
choice of the isomorphism of $W$ and a fiber of the
universal family $\pi _{X[n]}: X[n]^+\ra X[n]$.

\subsection{Trees of FM degeneration spaces} Notice that an FM
space of $X:=\PP ^1_{\bk}$ without markings over ${\bf
k}$ is a connected genus $0$ nodal curve with a distinguished
component $\PP ^1_{\bk}$. In that case there is a natural dual
tree graph corresponding to the prestable curve. Likewise, for an
FM space $W$ of a nonsingular variety $X$, we
associate a tree whose vertices (resp.\ edges) correspond one-to-one
to components of $W$ (resp.\ $W^{\mathrm{sing}}$); see Figure \ref{figureone}.
The distinguished vertex corresponding to the component, a blowup of
$X$, is called the {\em root} of the tree.
For a given vertex, the number of edges
of the minimal chain connecting the root and the vertex will be
called the {\em level} of the corresponding component of $W$. Let us call
a component of $W$ an {\em end} if it is not the root component
and the valance of its vertex is 1. A component of $W$ will be called
{\em ruled} if it is not the root and its vertex has valance 2.
The descendants of a component with associated vertex are by definition
the components of higher level for which the minimal chain
connecting the root and the associated vertex contains $v$. For each
non-root component $Y$ of $W$, we denote by $D_+(Y)$ the divisor of $Y$
corresponding to the edge which connects the vertex of $Y$ to the
vertex at lower level. The $+$ sign in $D_+(Y)$ is explained by
the fact that the intersection number of any curve on $Y$ and the
divisor is strictly positive unless the curve lies also on the
component one level higher.
 Here the
intersection number is taken in the nonsingular variety $Y$.
Similarly, we define $D_-(Y)$
when $Y$ is a ruled component. Non-root components of $W$ will
also be called {\em screens}.

\begin{figure}
\begin{pspicture}(-.25,1.75)(6.5,6.5)
\psline(0,5)(0.5,5)\psline(2,5)(2.5,5) \psline(2.5,5)(2.5,6.5)
\psline(2.5,6.5)(0,6.5) \psline(0,6.5)(0,5)
\psline(0.5,5.5)(2,5.5) \psline(2,5.5)(2,4)
\psline(2,4)(0.5,4)\psline(0.5,4)(0.5,5.5) 
\psline(0.5,2.75)(0.5,4)
\psline(1.25,2.5)(1.4,2.5)
\psline(2,3)(2,5.5)
\psline(0.75,3)(1.25,2.75)\psline(1.25,2.75)(1.25,1.75)\psline(0.75,3)(-0.25,2)
\psline(1.25,1.75)(-0.25,2)
\psline(1.75,3.25)(1.4,2.75)\psline(1.4,2.75)(1.4,1.75)\psline(1.75,3.25)(2.75,2.25)
\psline(1.4,1.75)(2.75,2.25)
\cnode[fillstyle=solid,fillcolor=black](6,6){.2}{A}\cnode[fillstyle=solid,fillcolor=black](6,4.8){.2}{B}\ncline{A}{B}
\cnode[fillstyle=solid,fillcolor=black](6,3.5){.2}{C}\ncline{B}{C}
\cnode[fillstyle=solid,fillcolor=black](5.2,2.5){.2}{E}\cnode[fillstyle=solid,fillcolor=black](6.8,2.5){.2}{F}
\ncline{C}{E}\ncline{C}{F}
\uput[0](6.2,6){Root}\uput[0](6.2,4.8){Ruled}\uput[0](5.4,2.5){End}\uput[0](7,2.5){End}
\uput[0](0,6){$\mathrm{Bl}_pX$} \uput[0](0.5,4.5){\small
$\mathrm{Bl}_{q}\PP ^r$} \uput[0](0.5,3.5){\small
$\mathrm{Bl}_{q',q''}\PP ^r$} \uput[0](0.25,2.1){\small $\PP ^r$}
\uput[0](1.4,2.1){\small $\PP ^r$}
\end{pspicture}
\caption{An FM degeneration space and its corresponding tree
graph.}
\label{figureone}
\end{figure}
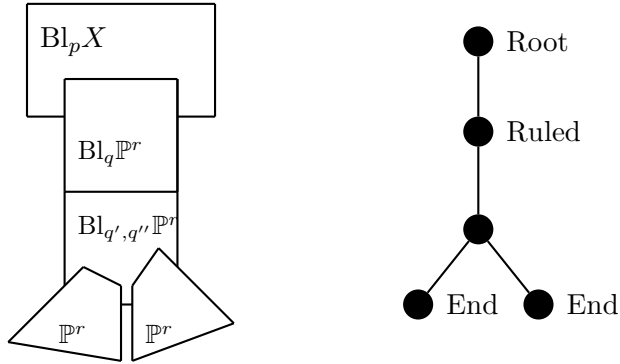

\subsection{Automorphisms} An FM degeneration $W$ of $X$ over $\Spec \bk$
is called a stable degeneration in \cite{FM}
since every  ruled (resp.\ end) screen contains at least one (resp.\ two) marked point(s).
We will forget about markings, and consider the automorphism groups that
then arise.
Define an automorphism of $W/X$ to be an automorphism $\psi$ of $W$ fixing
$X$, that is, $\pi _{W/X}\circ \psi = \pi _{W/X}$. Since there are
no markings in $W$, the automorphism group $\mathrm{Aut}(W,X)$
is nontrivial if there is an end
component. For example, the group is $\mathbb{G}_a^r\rtimes \mathbb{G}_m$
if the tree of $W$ has only two vertices. Note that
each end (resp.\ ruled) component $Y$ is (non-canonically)
isomorphic to $\PP ^r_{\bk}$ (resp.\ the blowup of $\PP ^r_{\bf k}$
at a point) with a marked divisor $D_+(Y)$ (resp.\ $D_{\pm}(Y)$).
The automorphism group of an end (resp.\ ruled) $Y$
fixing $D_+(Y)$ (resp.\ $D_{\pm}(Y)$) is a subgroup, isomorphic to
$\mathbb{G}_a^r\rtimes \mathbb{G}_m$ (resp.\ $\mathbb{G}_m$), of
$\mathrm{Aut}(W,X)$, where notation $\mathbb{G}_m$ (resp. $\mathbb{G}_a$) is 
for the multiplicative (resp. additive) group $\mathbf{k}^\times$
(resp. $\mathbf{k}$). So, for general
$(W,X)$, we see that
\[  \mathrm{Aut}(W,X)
\cong \bigg(\prod _{W_i :\text{ ruled} }\mathbb{G}_m
\bigg) \times \bigg(\prod _{W_i  : \text{ ends}}
\mathbb{G}_a^r\rtimes\mathbb{G}_m\bigg)\]

More generally, for an FM degeneration space $\cW$ of $X$ over $S$
we may analogously define $\mathrm{Aut}(\mathcal{W}/S,X)$, at least as
a presheaf of groups.
In case $\cW\to S$ is projective, one may use relative Hilbert schemes
(for example, see \cite{Kol}) to conclude that $\mathrm{Aut}(\mathcal{W}/S,X)$
is represented by a group scheme over $S$.
It is at least an algebraic space in general,
as a consequence of the results in the rest of this section.

\subsection{Operations}\label{operation}
\noindent{\em Forgetting Markings.} Since only the labeling of
points matters, we also use notation $X[M]$ for $X[m]$ if $M$ is a set
of cardinality $m$. For example, $X[N]$ will be used instead of
$X[n]$, where $N=\{1,\ldots,n\}$. There is a natural iterated blow-down map $X[N]$ to $X[J]\times X^{N\setminus J}$
if $J\subset N$. By similar reasoning, there is a natural  blow-down map
$X[N]^+ \ra X[J]^+\times X^{N\setminus J}$. To see the latter, we may assume that $J=N-1$, that is,
$\{1,\ldots,n-1\}$. Then, by the results of L. Li on rearrangements of centers \cite{L.Li},
$X[N]^+$ coincides with the iterated blowup of $X[J]^+\times X$ along
$\Delta _T$, $n\in T$, $|T|\ge 2$, where $|T|\ge 3$ whenever $n+1\in T$.
($X[N]^+$ is an iterated blowup of $X^J\times X^2$ along centers which form
a building set
$\{\Delta_T\,:\,T\subset N+1,\,T\ne \{a,n+1\},\,a\in N\}$.
Then using a building set order satisfying
$\Delta_{T_1}\prec \Delta_{T_2}$ if $n\in T_2\setminus T_1$ we see that
$X[N]^+$ is the iterated blowup of $X[J]^+\times X$ along
$\Delta _T$, $n\in T$, $|T|\ge 2$, where $|T|\ge 3$ whenever $n+1\in T$.)
Combined with projections, we obtain natural forgetful maps
$X[N]\ra X[J]$ and $X[N]^+\ra X[J]^+$.


By forgetting a point labeled by, say, $n+1$, there is a natural
commutative diagram
\[
\xymatrix{
X[N+1]^+ \ar[r]^(0.55){\pi_+} \ar[d]^{\pi_{X[N+1]}} &
X[N]^+ \ar[d]^{\pi_{X[N]}} \\
X[N+1] \ar[r]^(0.55)\pi & X[N]
}
\]
It induces a map $\tilde{\pi}_{+}: X[N+1 ]^+ \ra X[N+1 ]\times_{X[N]} X[N]^+$.

\begin{Lemma}\label{forgetting} The map $\tilde{\pi}_+$ is an
isomorphism over the open locus of $X[N+1 ]$ where $\pi _{+} \circ
\sigma_{n+1}$ meets neither any $\sigma _i$, $i=1$, $\ldots$, $n$ nor the
relative singular locus of $X[N]^+/X[N]$.
\end{Lemma}

\begin{proof}
We note that $\Delta _{i^+}$  in $X[N+1]^+$ coincides with
the proper transform of $\Delta _{i^+}$ in $X[N]^+\times X$
for $i=1$, $\ldots$, $n$. Hence
$\pi _+\circ \sigma _i = \sigma _i \circ \pi $,
and the result is established, for otherwise there would be an unstable
locus in $X[N+1]^+/X[N+1]$, which would be a contradiction.
\end{proof}


\noindent{\em Lifting.}
Let $g$ be a map from $S$ to $X[n]$ and let $h$ be a section of $g^*X[n]^+$.
Assume that the image of $h$ meets neither the relative singular
locus of $X[n]^+/X[n]$ nor any sections $\sigma _i\circ g$, $i\in N$.
Then there is a unique lift $\tilde{h}:S\ra X[n+1]$ of $h$, since
$X[n+1]$ is the blowup of $X[n]^+$ along the images of sections
$\sigma _i$.
We note that the $S$-scheme $g^*X[n]^+$ is canonically
isomorphic to the $S$-scheme $\tilde{h}^*X[n+1]^+$ preserving $\sigma
_i$, $i=1$, $\ldots$, $n$, due to Lemma \ref{forgetting} and the diagram

\[
\xymatrix{
     &   X[n+1] \ar[d]     & X[n+1]^+ \ar[dd] \ar[l] \\
     &   X[n]^+ \ar [d]    \ar@{=}[rd] & \\
S \ar[r]_g \ar[ur]_h \ar@{-->}[uur]^{\tilde{h}} & X[n]  & X[n]^+. \ar[l]}
\]
This diagram is commutative except the trapezoid. The trapezoid however
commutes if it is restricted to the image $\Delta _{\{n+1,n+2\}}$ of the section $\sigma _{n+1}$.
This implies that $\sigma _{n+1}\circ \tilde{h}$ coincides with $h$ when
$\tilde{h}^*X[n+1]^+$ is identified with $g^*X[n]^+$.
 The iterated
operation of liftings will be a key tool in Section
\ref{MainSection}.

\begin{Prop}\label{univ} Let $f_i:S\ra X[n]$ be a morphism
of  $\bk$-schemes. Suppose that there is an isomorphism
between $f_i^*X[n]^+$, $i=1$, $2$, fixing $X$ and preserving the $n$ induced
sections. Then $f_1=f_2$.
\end{Prop}
\begin{proof}
This can be seen by the universal property
of Theorem 4 in \cite{FM} and the natural identification
$\mathcal{I}_{\Delta _J} |_x \cong ((T_xX)^J/T_x)^*$, where
$\mathcal{I}_{\Delta _J}$ is the ideal sheaf of diagonal $\Delta
_J$ of $X^J$ and $x\in X=\Delta _J \subset X^J$,
for $J\subset N$, $|J|\ge 2$. (See page 195 of
\cite{FM} for a friendly explanation of the universal property and
its relation with screens.)
Let $f$ denote $f_i$ followed by the map $X[N]\ra X^N$.
Then, \'etale locally at a point which is mapped into $\Delta _S$ in $X^J$,
we can express the data $f^*\mathcal{I}_{\Delta _J} \ra
\mathcal{L}_J\cong \mathcal{O}_S$ by sending $x_{i,j}-x_{i',j}$ to $\sigma
_{i,j}-\sigma _{i',j}$ for $i$, $i'\in J\subset N$, $j=1$, $\ldots$, $r$,
where $x_{i,j}$ are (copies of) coordinates of the
$i$-th component $X$ of $X^n$, and $\sigma _{i,j}$, $j=1$, $\ldots$, $r$
are fiberwise direction coordinates of $X[J]^+$,
pulled back via the composite of
the sections $\sigma_i$ with the forgetting map
$$\tilde\pi_+\colon X[N]^+\to X[N]\times_{X[J]}\mathrm{Bl}_{\Delta _{J^+}} (X[J]\times X).$$

We are thus reduced to showing that the given isomorphism induces an
isomorphism between the $f_i^*(X[N]\times_{X[J]}\mathrm{Bl}_{\Delta _{J^+}} (X[J]\times X))$, $i=1$, $2$.
This follows from
the claim that for any $S\to X[N]$ the following base change property holds:
\begin{align*}
(1_S\times \tilde\pi_+)_*\mathcal{O}_{S\times_{X[N]}X[N]^+}&\cong
\mathcal{O}_{S\times_{X[J]}\mathrm{Bl}_{\Delta _{J^+}} (X[J]\times X)},\\
R^p(1_S\times \tilde\pi_+)_*\mathcal{O}_{S\times_{X[N]}X[N]^+}&=0\qquad
\text{for $p>0$}.
\end{align*}
(The assertion about direct image would suffice, but the following
reduction step uses also the vanishing of the higher direct images.)
It suffices to treat the case $J=N-1$, and
since $\tilde\pi_+$ is a proper morphism of flat
$X[N]$-schemes, a cohomology-and-base-change argument
as in \S{}II.5 of \cite{Mu} allows us to reduce to the case $S=\Spec \bk$.
Then $1_S\times \tilde\pi_+$ is (when it is not an isomorphism)
a proper morphism of FM degeneration spaces contracting a ruled
component or an end component.
In either case, the argument proceeds using the Theorem on Formal Functions:
for the direct image, by a computation in local coordinates, and for
the higher direct images, as in the proof of Theorem I.9.1(ii) of \cite{BHPV}.
\end{proof}

\subsection{Spaces $X[n:m]$ and $X[n:m]^+$}
In this subsection we define a compactification $X[n,m]$ of the
configuration space of pairs of black colored $n$ ordered points
in $X$ and red colored $m$ ordered points in $X$. Let $N+M$ denote
$\{1_b,\ldots,n_b\}\sqcup \{1_r,\ldots,m_r\}$, a collection of $n$
``black numbers'' and $m$ ``red numbers''. We will call the
labels of $N$ black and the labels of $M$ red. Then we want
to compactify
\[ \bigg( X^{N}\setminus \bigcup _{B\subset N \ : \ |B|=2}\Delta _B \bigg)\times
\bigg( X^M\setminus \bigcup _{R\subset M \ : \ |R|=2 } \Delta _R
\bigg)
\] allowing red points to collide with black points but not
allowing any two points with the same color to coincide. A
``universal family'' can be constructed by an iterated blowup of
$X[n,m]\times X$ along smooth centers. The centers will be proper
transforms of suitable diagonals in the morphism $X[n,m]\times
X\ra X^n\times X^m\times X$. Those diagonals are
\begin{itemize}

\item $\Delta _{I^+}$ with
$|I_{\mathrm{black}}||I_{\mathrm{red}}|\ge 2$,
$I_{\mathrm{black}}:=I\cap N$ and $I_{\mathrm{red}}:=I\cap M$;

\item $\Delta _{B^+}$ with $|B|\ge 2 $, $B\subset N$;

\item $\Delta _{R^+}$ with $|R|\ge 2$, $R\subset M$,

\end{itemize}
where $A^+=A\cup \{n+m+1\}$.

A detail of the construction is as follows. First, let
$X[1,1]=X^2$. Now $X[n,m]$ and its universal family will be
defined inductively. Start with $X[n,m] \times X$ and blow it up
along $\Delta _{I^+}$ and then $\Delta _{B^+}$ and $\Delta
_{R^+}$. By \cite{L.Li}, any order of the blowups along centers of each type will
give the same outcome as long as the blowup along $\Delta _{I^+_1}$
is taken before the blowup along $\Delta _{I^+_2}$ whenever
$|I_1|> |I_2|$. Denote by $X[n,m]^+$ the result of the blowups.
Then $\Delta _{i_b^+}$ and $\Delta _{i_r^+}$ provide sections
$\sigma _{i_b}$ and $\sigma _{i_r}$ of $X[n,m]^+\ra X[n,m]$. Now
define $X[n,m+1]$ to be the blowup of $X[n,m]^+$ along all $\Delta
_{\{b,r\} ^+}$ and then all $\Delta _{r^+}$. We mark the last
label as $(m+1)_r$. Similarly, define $X[n+1,m]$ to be the blowup
of $X[n,m]^+$ along all $\Delta _{\{ b,r\}^+}$ and then all
$\Delta _{b^+}$. We mark the last label as $(n+1)_b$. In fact,
$X[n,m]$ is the closure of
\[ X^{n+m}\setminus \bigcup _{\Delta} \Delta \] in
\[ X^{n+m} \times \prod _{\Delta} \mathrm{Bl}_{\Delta} X^{n+m} \] where
$\Delta$ runs over all diagonals except those of type
$\Delta _{b,r}$, $b\in N$ and $r\in M$. The
claims above can be justified by directly modifying the arguments in
\cite{FM} or by L. Li's general approach to the wonderful
compactification \cite{L.Li}.

Define $X[n:m]$ to be the maximal
open subset of $X[n,m]$ such that the restriction of the universal
family $X[n,m]^+$ to the subset is still stable after forgetting
red markings but keeping the black markings and vice versa.
Here
``stable'' means by definition that every fiber has only the
trivial automorphism fixing $X$ and the remaining marked points.
Denote by $X[n:m]^+$ the restriction of the fibration $X[n,m]^+\ra
X[n,m]$ to $X[n:m]$.

\begin{Lemma}\label{XnmUniv}
The image of a map $g: S\ra X[n_1,n_2]$ is in $X[n_1:n_2]$ if and
only if $g^*X[n_1,n_2]^+$ is isomorphic to $g_i^*X[n_i]^+$
preserving $n_i$-sections and fixing $X$, for $i=1, 2$, where $g_i$
is the composite $S \stackrel{g}{\rightarrow} X[n_1,n_2]\ra
X[n_i]$.
\end{Lemma}

\begin{proof} For $i=1, 2$,
there is the naturally induced map \[ h_{i}: X[n_1,n_2]^+\ra
X[n_1,n_2]\times _{X[n_i]}X[n_i]^+
\] over $X[n_1,n_2]$. Then $X[n_1:n_2]$ is the maximal
open subset of $X[n_1,n_2]$ over which $h_{1}$ and $h_{2}$ are
isomorphisms.
\end{proof}

\begin{Prop}\label{PairOper}
Consider $g_1: S\ra X[N]$, $N=\{1_b,\ldots,n_b\}$
with extra sections $\sigma_i$,
$i\in M=\{1_r,\ldots,m_r\}$, of $g_1^*X[N]^+$ such that the extra sections
meet  neither each other nor the
relative singular locus of $g_1^*X[N]^+/S$.
Then it induces a unique map $g: S\ra X[N,M]$ such that
canonically, $ g^*X[N,M]^+\cong g_1^*X[N]^+$ preserving
$N$-labeled sections and fixing $X$; the extra sections
coincide with the sections from the second label $M$.

Furthermore, if each geometric fiber of $g_1^*X[N]$ is stable with respect to
markings by $\sigma_i$, $i\in M$, then canonically
$g^*X[N,M]^+\cong g_2 ^*X[M]^+$ preserving $M$-labeled
sections and fixing $X$, where $g_2$ is the natural composite
$S\stackrel{g}{\ra} X[N,M]\ra X[M]$.
\end{Prop}

\begin{proof}
The proof of the first statement follows from the inductive use of
the argument similar to the one given in \S\ref{operation}.
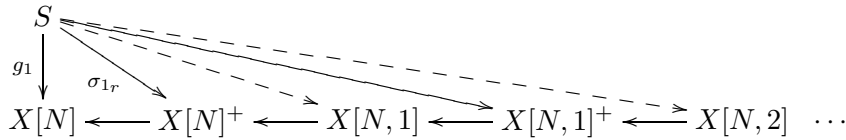
\[ \xymatrix{
S  \ar[d]_{g_1} \ar[rd]_{\sigma _{1_r}} \ar@{-->}[rrd] \ar[rrrd] \ar@{-->}[rrrrd]  &  &  &      &  \\
X[N] & X[N]^+ \ar[l] & X[N,1] \ar[l] & X[N,1]^+ \ar[l] & X[N,2] \ar[l] \ \  \cdots
}\]
The second statement is an
immediate consequence of Lemma \ref{XnmUniv}.
\end{proof}

\subsection{Definition of $\X [n]$}
Let $\X[n]$ denote the stack of FM spaces of level $n$ over
$(\mathrm{Sch}/\bk )$.
It follows from Propositions \ref{univ} and \ref{PairOper} that
$\X[n]\cong [ X[n:n] \rightrightarrows  X[n]]$, the stackification
of the prestack associated to a groupoid scheme
\[ X[n:n] \mathrel{\operatornamewithlimits{\rightrightarrows}_s^t} X[n].\]
The groupoid scheme is equipped with obvious maps $s$ and $t$,
diagonal map $e:
X[n]\ra X[n:n]$, ``composition'' $m: X[n:n]\,
{}_{t}\!\times _{s} X[n:n] \ra X[n:n]$, and exchange
map $i: X[n:n]\ra X[n:n]$. The stack is algebraic and smooth by
Proposition 4.3.1 in \cite{LM}, which requires that
$s$ and $t$ are smooth and the relative diagonal $(s,t): X[n:n]\ra X[n]\times
X[n]$ is separated and quasi-compact; these conditions are easily checked.

There is also a variant $X[n:n]_m$ of $X[n:n]$ for $m\le n$,
defined as a subscheme by the condition
$\sigma_{i_b}=\sigma_{i_r}$ for $i=1$, $\ldots$, $m$, or
alternatively by a blowup construction.
This leads to the stack
$\X[n]_m\cong [X[n:n]_m\rightrightarrows X[n]]$ of
FM spaces with $m$ sections, locally defined by the given $m$ sections
plus $n-m$ additional sections.
There is the forgetful morphism $\X[n]_m\to X[m]$.

In the particular case $n=m+1$, there is another variant
$\X[m+1]'_m$ of FM spaces with $m$ sections and marked component,
locally defined by the given sections plus one more lying on the
marked component.
As a groupoid,
$\X[m+1]'_m\cong [X[m+1:m+1]'_m\rightrightarrows X[m+1]]$
where $X[m+1:m+1]'_m$ is the open subscheme of
$X[m+1:m+1]_m$ where the black and red $(m+1)$-st sections lie on the
same component.

There is, further, the open substack $\X[m+1]''_m$ of $\X[m+1]'_m$ where
the marked component is that of the $m$-th section.
It is also an open substack of $\X[m+1]_m$, and if we define
$X[m+1]''$ to be the locus in $X[m+1]$ where the $m$-th and
$(m+1)$-st sections lie on the same component,
with $X[m+1:m+1]''_m$ the common pre-image in $X[m+1:m+1]'_m$, then
$\X[m+1]''_m\cong [X[m+1:m+1]''_m\rightrightarrows X[m+1]'']$.
A further open substack is
$X[m]\cong [X[m+1:m+1]'''_m\rightrightarrows X[m+1]''']$
where the FM space with $m$ sections is itself stable.
Here $X[m+1]'''$ is the complement in $X[m+1]''$ of the
divisor $\Delta_{m,m+1}$ where the $m$-th and $(m+1)$-st sections have come together.
We will also use the forgetful morphism
$\X[m+1]'_m\to \X[2]_1$ forgetting the first $m-1$ sections, given via
groupoids as
\[ [X[m+1:m+1]'_m\rightrightarrows X[m+1]]\to
[X[2:2]_1\rightrightarrows X[2]], \]
as well as the isomorphism $\X[m+1]''_m\cong X[m]\times [\A^1/\mathbb{G}_m]$
given by forgetful morphism (first factor) and divisor
$\Delta_{m,m+1}\subset X[m+1]''$ mentioned above (second factor).
The multiplicative group $\mathbb{G}_m$ acts in the standard way on $\A^1$ with quotient stack
$[\A^1/\mathbb{G}_m]$ parametrizing pairs consisting of a line bundle
with a regular section;
note that an effective Cartier divisor canonically determines such a pair.

\section{The Stack of Stable Ramified Maps}

\subsection{Stable ramified maps} We introduce a generalized notion of
stable maps. For the basic definitions and properties of stable
maps, the reader may see, for example, \cite{FP}. From now on, we assume that $X$ is a
nonsingular projective variety over $\bk$. Let \[ NE_1(X)\subset
A_1(X )/\sim _{\mathrm{alg}}\] denote the semigroup of effective curve classes modulo algebraic equivalences.
Given $\beta \in NE_1(X)$, $g,n\in \ZZ _{\ge 0}$, and $\mu =(\mu
_1,\ldots,\mu _n)$, $\mu _i\in \ZZ _{\ge 1}$, we define:

\begin{Def}\label{DefPmap} A triple \[((C, p_1,\ldots,p_n), \pi _{W/X}: W\ra X, f:C\ra
W)\]
      is called a {\em stable map
       with $\mu$-ramification from an $n$-pointed, genus $g$ curve to an FM degeneration space $W$ of
       $X$, representing class
       $\beta$}
       (for short, {\em a $(g,\beta,\mu)$-stable ramified map})
       if:

\begin{itemize}
       \item $(C,p_1,\ldots,p_n)$ is an $n$-pointed, genus $g$
       prestable curve over $\bk$.

       \item $\pi _{W/X}: W\ra X$ is an FM degeneration of $X$ over
       $\bk$.

       \item  The pushforward $(\pi _X\circ f)_*[C]$ of the fundamental class $[C]$ is
       $\beta$.

        \item The following four conditions are satisfied:
\end{itemize}

\begin{enumerate}

       \item {\em Prescribed Ramification Index Condition:}
     \begin{itemize}
       \item The smooth locus $C^{\mathrm{sm}}$ of $C$ coincides with
       the inverse image 
       $f^{-1}(W^{\mathrm{sm}})$ of the smooth locus $W^{\mathrm{sm}}$.
       \item
       $f|_{C^{\mathrm{sm}}}$ is unramified everywhere possibly except at
       $p_i$.
       \item At $p_i$
       the ramification index
       \[ \mathrm{length}(\mathfrak{m}_{p_i} /\mathfrak{m}_{f(p_i)}\mathcal{O}_{p_i}) +1\]
       of $f$ is exactly $\mu _i$.
       \end{itemize}

       \item {\em Distinct Points Condition:} $f(p_i)$, $i=1$, $\ldots$, $n$ are pairwise distinct
       points of $W$.

       \item {\em The Admissibility
       Condition:}
       At every nodal point $p$ of $C$, there are identifications
       $\widehat{\mathcal{O}}_{f(p)}\cong \bk[[z_1,\ldots,z_{r+1}]]/(z_1z_2)$
       and $\widehat{\mathcal{O}}_p\cong \bk[[x,y]]/(xy)$, so that
       $\hat f^*$ sends $z_1$ to $x^m$ and $z_2$ to $y^m$ for some
       positive integer $m$.

       \item {\em Stability Condition:}
        \begin{itemize} \item For each ruled component $W_r$ of $W$, there
       is either an image of a marking in $W_r$ (that is, $f(p_i)\in W_r$ for
       some $i$) or a non-fiber image $f(D)\subset W_r$ of an
       irreducible component $D$ of $C$.
       \item For each end component $W_e\cong \PP ^r$ of $W$, there are either
       images of two distinct markings in $W_e$ or a non-line image
       $f(D)\subset W_e$ of an irreducible component $D$ of $C$.
       \end{itemize}

\end{enumerate}

\end{Def}

\begin{Lemma}\label{boundedness} Let $W$ be a target
of a $(g,\beta,\mu)$-stable ramified map $f:C\ra W$. Then the
number of components of $W$ is bounded above by an integer
depending only on $X$, $\beta$, $g$ and $n$. So is the number of
components of $C$.
\end{Lemma}

\begin{proof}
First assume that $W=X$ is the projective line $\PP ^1_{\bk}$, $C$ is a genus $g$
nodal curve not necessarily connected, $d$ is a positive integer, and $f:C\ra W$ is a degree
$d$ map whose restriction to each connected component is stable.
Then the sum of the number of nodal points of $C$,
ramification points of $f$, and marked
points is less than or equal to $2d-2+2g+2n$.

Now we come to the general case that $W$ is an FM-degeneration
space of a nonsingular projective variety $X$. Choose an embedding
$X\hookrightarrow \PP ^N_{\bk}$. Let $W_k$, $W_{\ge k}$ be the set of all level-$k$
resp.\ level-$k$-and-higher screens,
$C_{\ge k}$ the pre-image of $W_{\ge k}$ under $f$,
and $C_k$ the stabilization of $C_{\ge k}\to W_{\ge k}\to W_k$.
By considering a general
projection to $\PP^1_{\bk}$ such that the composite
$C_0\to W\ra X \dashrightarrow \PP ^1_{\bk}$ is well-defined, stable, and
collapses no components not already collapsed by the map to $X$,
we see that $|W_1|  \le 2d-2+2g+2n$,
where $d$ is the degree of $\beta$. Note that the total degree of $C$ in {\em
each} level-one screen (understood as degree in $\PP^r_{\bk}$ after blowdown) is less
than or equal to $d$. By the Admissibility Condition and the Stability
Condition, the sum of the total degree and the number of marked
points of $C$ in {\em each} level-two screen is less than or equal
to $d+n-1$. By the same reasoning, the sum of the total degree and the number
of marked points of $C$ in a level-$k$ screen is less than or
equal to $d+n-k+1$.
Therefore the level of any screen can be at most $d+n$.
This fact combined with the bound
$|W_k|\le 2(d-1+g+n)\prod_{1\le i\le k-1} 2(d+g+n-i)$ leads to a bound
on the total number of components of $W$, and since each
component supports the image of at most $d$ components of $C$,
also a bound on the number of components of $C$.
\end{proof}

\subsection{The families of stable ramified maps}
Given $\beta \in NE_1(X)$, $g,n\in \ZZ _{\ge 0}$, and $\mu =(\mu
_1,\ldots,\mu _n)$ where $\mu _i\in \ZZ _{\ge 1}$, we define a family
version of  stable ramified maps.

\begin{Def}\label{definition} A triple \[ ((\pi :\mathcal{C}\ra S, \{ p_1,\ldots,p_n\}),
(\pi _{\mathcal{W}/S}: \mathcal{W} \ra S, \pi _{\mathcal{W}/X}:
\mathcal{W}\ra X), f:\mathcal{C}\ra \mathcal{W})\] is called an
{\em $S$-family of stable maps with $\mu$-ramification from $n$-pointed,
       genus $g$ curves to an FM degeneration space $\mathcal{W}$ of
       $X$, representing class $\beta$} if:

\begin{itemize}
       \item $(\pi :\mathcal{C}\ra S, \{ p_1,\ldots,p_n\})$ is a family of $n$-pointed, genus $g$
       prestable curves over $S$.

       \item $(\pi _{\mathcal{W}/S}: \mathcal{W} \ra S,
       \pi _{\mathcal{W}/X}: \mathcal{W}\ra X)$ is an FM degeneration of $X$ over
       $S$.

       \item  The data form a commutative
       diagram
       \[\xymatrix{
       \mathcal{C} \ar[rr]^f \ar[dr]_{\pi} &&
       \mathcal{W} \ar[r]^{\pi_{\mathcal{W}/X}} \ar[ld]^{\pi_{\mathcal{W}/S}} &
       X\\ & S }
       \]
         such that over each geometric point of $S$,
         it is a $(g,  \beta,\mu  )$-stable ramified map.

       \item {\em Prescribed Ramification Index Condition:} $f$ has ramification index $\mu _i$ at $p_i$.

       \item {\em The Admissibility
       Condition:}  For any geometric point $t$ of  $S$, if $p$ is a nodal point of $C_t$ and two isomorphisms are given as
        \[ \widehat{\mathcal{O}}_{f(p)}
       \cong \widehat{\mathcal{O}}_{\pi _S (p)} [[z_1,\ldots,z_{r+1}]]/(z_1z_2-s), \text{ for some }   s\in \widehat{\mathcal{O}}_{\pi _S(p)}\] and 
         \[ \widehat{\mathcal{O}}_{p}\cong\widehat{\mathcal{O}}_{\pi_S (p)}[[x,y]]/(xy-s'), \text{ for some } s'\in \widehat{\mathcal{O}}_{\pi _S(p)}\] 
      then
             \[\hat{f}^*(z_1)= \alpha _1 x^m,\ \hat{f}^*(z_2) = \alpha _2 y^m \]
       for some
       units $\alpha _i$ in $\widehat{\mathcal{O}}_{p}$
       with $\alpha _1\alpha _2 \in  \widehat{\mathcal{O}}_{\pi _S(p)} $
        and a positive integer $m$.
\end{itemize}
\end{Def}

\begin{Rmk} \label{RmkRam} Let $\mathcal{C}$ and $\mathcal{W}$ be as in above.
Then we say that $f:\mathcal{C}\ra \mathcal{W}$ has ramification
index $\ell$ at a smooth point $p: S\ra \mathcal{C}$ if
we have equality of sheaves of ideals
\[ f^*(\mathcal{I}
_{f(p(S))})\widehat{\mathcal{O}}_{p(s)} = \mathcal{I}_{p(S)}^\ell\widehat{\mathcal{O}}_{p(s)}
\]
for all $s\in S$.
\end{Rmk}

\begin{Rmk} Admissibility Condition,
which is called Predeformability Condition in
\cite{J.Li}, was introduced and studied by J. Li in his
construction of stable relative map and relative Gromov-Witten
invariants. When the target is one-dimensional, the notion of
admissibility was introduced in \cite{HM} and well studied in
\cite{Mo} by log structures (\cite{Kato});
the analogous log structures in case
of higher-dimensional target are studied in \cite{Olsson2}. As explained in \S 3.7 in \cite{Mo} and
in Simplification 1.7 in \cite{J.Li2},
we may let $\alpha _1=\alpha _2 =1$ in Definition \ref{definition} for a suitable isomorphism 
$\widehat{\mathcal{O}}_{p}\cong \widehat{\mathcal{O}}_{\pi_S(p)}[[x,y]]/(xy-s')$.
\end{Rmk}

From now on, we use abbreviations of the imposed conditions on
stable ramified maps, for example, PRIC for
Prescribed Ramification Index Condition.

\begin{Lemma}\label{TLMC}{\em (Tangent Line Map Condition)}
Let $f$ be a family of stable ramified maps as in Definition
\ref{definition}.
PRIC and AC together imply that there is a natural extension
\[ \PP (Tf):\mathcal{C} \rightarrow
   \PP(TX)\] of the projectivization
       of the induced map $T(\mathcal{C}/S)^{\mathrm{sm}}\ra T(\mathcal{W}/S)^{\mathrm{sm}}$ between tangent bundles.
\end{Lemma}
\begin{proof} This is a local property. PRIC (resp.\ AC)
will imply that $\PP (Tf)$ is well-defined at smooth points (resp.\ at
singular points) of $\pi : \mathcal{C}\ra S$. First, at an
smooth point $p$ of $\mathcal{C}/S$, let $x$ be a $S$-relative
uniformizing parameter, locally defining the section $p_i$ if $p=p_i$ for
some $i$. Since the image of $p$ is a smooth point
$q$, we have also $t_1,\ldots,t_r$, $S$-relative uniformizing
parameters at $q$. The submodule $(f^*dt_1,\ldots,f^*dt_r)$ of the stalk
of $\Omega _{\mathcal{C}/S}$ at $p$ is
generated by $x^{m-1}dx$ if $m$ is the ramification index of $f$ at
$p$. Hence
$$[x^{-(m-1)}f^*dt_1,\ldots,x^{-(m-1)}f^*dt_r]$$
is regular at $p$ since there is no jump of the ramification indices at $p$.
In other words,
$f^*\Omega _{\mathcal{W}/S}\ra \Omega_{\mathcal{C}/S}(-\sum (m_i-1)p_i)$
is surjective on the smooth locus of $\mathcal{C}/S$.

At a singular point of $\mathcal{C}/S$, the map $f$ satisfies AC:
\[  A [[z_1,\ldots,z_{r+1}]]/(z_1z_2-s) \ra A[[x,y]]/(xy-s') \]
sending $z_1\mapsto \alpha _1 x^m,\ z_2\mapsto \alpha _2 y^m$, where
$A=\widehat{\mathcal{O}}_{\pi (p)}$. Using the local description
of $\pi _{X[n]}$ and the projective tangent map given in \S\ref{local-description}, we may assume that $W=Y_{k+1}$
and $F_k=X$;
and we compute that
$\PP (Tf)$ is
$$[m,mz_3+x\frac{\partial z_3}{\partial x},\ldots,mz_{r+1}
+x\frac{\partial z_{r+1}}{\partial x}]$$
for $x\ne 0$ and
$$[m,mz_3-y\frac{\partial z_3}{\partial y},\ldots,mz_{r+1}
-y\frac{\partial z_{r+1}}{\partial x}]$$
for $y\ne 0$.
Since $x\,\partial/\partial x=-y\,\partial/\partial y$ is a regular
derivation of $A[[x,y]]/(xy-s')$, it follows that
the rational map $\PP Tf: \mathcal{C}\dashrightarrow \PP TX$  is
well-defined also at nodal points.
\end{proof}

\subsection{Construction of $\overline{\mathfrak{U}}_{g,\mu}(X,\beta )$}
Define a category $\overline{\mathfrak{U}}_{g,\mu}(X,\beta)$
of $(g, \beta , \mu)$-stable ramified maps to FM degenerations of $X$; it
is a CFG over the \'etale site $(\mathrm{Sch}/\bk )$. A morphism
is a commutative diagram
\[
\xymatrix@R=14pt{
\mathcal{C}'  \ar[rr]^{f'}\ar[dr]  \ar[dd] && \mathcal{W}' \ar[dd]\ar[ld]\ar[r]&
X \ar@{=}[dd] \\
& S'  \ar[dd] \\
\mathcal{C}  \ar'[r][rr]^(.3){f} \ar[dr]  &&  \mathcal{W} \ar[ld]\ar[r] &X \\
& S
}\]
preserving markings, where the squares to each side of $S'\ra S$ are cartesian.
It is straightforward to see that this CFG is a stack.

\begin{Def}\label{W/T}
An $S$-family of {\em stable $\mu$-ramified maps to a fixed target}
$X[\ell]^+/X[\ell]$ from $n$-pointed genus $g$ curves is
a stable $n$-pointed genus $g$ map
$(\cC\ra S,\{ p_1,\ldots,p_n\}, f:\cC\ra X[\ell]^+)$ with a
commutative diagram
\[
\xymatrix{
\cC \ar[r]^(0.42)f\ar[d] \ar[d] & X[\ell]^+\ar[d] \\
S \ar[r] & X[\ell]
}\]
satisfying PRIC, DPC,  AC, and SC.
\end{Def}

Let $\beta$ be a curve class in $NE_1(X)$ and denote also by
$\beta$, the induced class in $NE_1(X[\ell]^+ )$ using any
canonical inclusion $X\subset X[\ell]^+$ as a general fiber of
$X[\ell]^+\ra X[\ell]$.

\begin{Prop}\label{AC}
The stack $\overline{\mathcal{M}}_{g,n}(X[\ell]^+/X[\ell],\beta  )^\mu$ of
stable $\mu$-ramified maps to a fixed target $X[\ell]^+/X[\ell]$ from
$n$-pointed genus $g$ curves, representing class $\beta $, is a
separated finite-type Deligne-Mumford stack over $\bk$.
\end{Prop}
\begin{proof} Since the moduli of $(g,\beta)$-stable maps to $X[\ell]^+$ is a proper
Deligne-Mumford stack, it is enough to show that, given a family $f$ of stable
maps over $S$, representing class $\beta$, there is a locally closed subscheme
$Z$ of $S$ such that: for any $T\ra S$, the pullback $f|_T$ is a
$T$-family of stable ramified maps (with fixed target
$X[\ell]^+/X[\ell]$) if and only if $T\ra S$ factors through $Z$.

First, take the maximal open locus $S_1$ of $S$ where $f$ does
not send any component of any geometric fiber of $\cC$ to the
relative singular locus of $X[\ell]^+/X[\ell]$, and furthermore DPC
is satisfied. Then there is a natural closed subscheme $S_2$
of $S_1$ representing a functor of admissible stable maps. This
can be shown by the proof of Theorem 2.11 in \cite{J.Li} (or one
may use \S 3C in \cite{Mo}). Now take the maximal open locus $S_3$
of $S_2$ where $f^*\Omega^\dagger _{X[\ell]^+/X[\ell]} \ra \Omega^\dagger
_{\cC _{S_3}/S_3}$ is surjective, where $\Omega^\dagger
_{X[\ell]^+/X[\ell]}$ and $\Omega^\dagger _{\cC _{S_3}/S_3}$ are the
sheaves of relative log differentials induced from the log
structures of the boundary divisors.
In order to take care of PRIC, introduce relative uniformizing
parameters $z_j$, $j=1$, $\ldots$, $r$ of $\mathcal{O}_{f(p_i(t))}$ at
$f(p_i(t))$ where $t$ is a point of an \'etale chart of $S_3$, and let $z$ be a
relative parameter of $\mathcal{O} _{p_i(t)}$ at $p_i(t)$. Then,
define the closed subscheme $S_{4, i}$ of $S_3$ by equations
$a_{0, j}=0$, $\ldots$, $a_{\mu_i-1, j}=0$, for all $j$, where
$f^*(z_j) = a_{0,j} + a_{1,j}z + \cdots \in \mathcal{O}_{p_i}$, $a_{k, j}\in
\mathcal{O}_{t}$.
Then take the maximal open subscheme $S_5$ of $\bigcap _i S_{4,i}$,
where restricted to every geometric fiber
$\cC _t$, $f$ has ramification order exactly $\mu _i$ at $p_i(t)$.
Now $Z$ is the maximal open subscheme of $S_5$ where SC is satisfied.
\end{proof}

\begin{Cor}
The stack $\overline{\mathfrak{U}}_{g,\mu }(X,\beta)$ is a finite-type
Deligne-Mumford stack.
\end{Cor}
\begin{proof}
We apply Theorem 4.21 in \cite{DM}, using that by Lemma \ref{boundedness}
for sufficiently large $\ell$ we have
$\overline{\mathfrak{U}}_{g,\mu }(X,\beta)\to \X[\ell]$ and an isomorphism
\[\overline{\mathcal{M}}_{g,n}(X[\ell]^+/X[\ell],\beta  )^\mu \cong
\overline{\mathfrak{U}}_{g,\mu }(X,\beta)\times_{ \X[\ell] } X[\ell].\]
We need to check conditions:

(1) The diagonal map of the stack is representable, quasi-compact,
separated and unramified.
Using the isomorphism, the first three properties follow from
Proposition \ref{AC} (see \cite[Lemma C.5]{AOV}); the last property follows from SC.

(2) There is a scheme $U$ and a smooth surjective map
$U\ra \overline{\mathfrak{U}}_{g,\mu }(X,\beta)$.
This again follows from Proposition \ref{AC}.
\end{proof}

\section{Properness}\label{MainSection}

\subsection{Preliminary results}
\begin{Lemma}
\label{generalpairsections}
Let $(f:\cC\to X[m]^+, p_1,\ldots,p_n)$ be a stable ramified map over $K$
with
partial stabilization $(\cC,p_1,\ldots,p_n)\to (\cC',p'_1,\ldots,p'_n)$
of prestable curves,
extension of the latter to
$(\overline{\cC}', \bar p'_1,\ldots,\bar p'_n)$ over $R$,
and chosen component $E\subset \overline{\cC}'_0$.
Then there exists open
$\mathcal{U}\subset \overline{\cC}'\times_{\Spec R} \overline{\cC}'$
whose special fiber $\mathcal{U}_0$ satisfies
$\emptyset\ne \mathcal{U}_0\subset E\times E$,
such that for any sections
$(\sigma_1, \sigma_2): \Spec R\to \mathcal{U}$,
adding these to obtain
$\Spec K\to X[m+2]$ and forgetting the first $m$ sections yields
$\cC\to X[m]^+|_{\Spec K}\cong X[m+2]^+|_{\Spec K}\to X[2]^+$ such that
the stabilization $(\cC''\to X[2]^+,p''_1,\ldots,p''_n,\sigma''_1,\sigma''_2)$
extends over $R$ to $\overline{\cC}''\to X[2]^+$
with the extended $\sigma''$-sections landing in the
same component of $\overline{\cC}''_0$, and
the induced $\overline{\cC}'\dashrightarrow X[2]^+$
restricts to a nonconstant map on $E$.
\end{Lemma}

\begin{proof}
We easily reduce to the case $\cC'=\cC$, then we write
$(\overline{\cC},\bar p_1,\ldots,\bar p_n)$ for
$(\overline{\cC}',\bar p'_1,\ldots, \bar p'_n)$.
There exists open
$\mathcal{V}\subset \overline{\cC}\times_{\Spec R} \overline{\cC}$
disjoint from the diagonal of $\overline{\cC}$, the sections
$\bar p_i$ on both factors, and the pre-image in
$\cC\times_{\Spec K}\cC$ of the diagonal of $X[m]^+$, the $m$ sections,
and the relative singular locus of $X[m]^+/X[m]$ on both factors,
with special fiber $\mathcal{V}_0$ nonempty and contained in $E\times E$,
such that there is a morphism
$$\mathcal{V}\to X[m+2]$$
extending the obvious one on $\Spec K\times_{\Spec R}\mathcal{V}$.
By symmetry the generic point of $\mathcal{V}_0$ maps into $X[m+2]''$,
so by further shrinking $\mathcal{V}$ preserving the condition
$\emptyset\ne \mathcal{V}_0\subset E\times E$ we may suppose $\mathcal{V}$
irreducible and smooth over $\Spec R$ with image contained in
$X[m+2]''$.
By construction the image of
$\Spec K\times_{\Spec R}\mathcal{V}$
is contained in $X[m+2]'''$, which means that the divisor
$\Delta_{m+1,m+2}$ of $X[m+2]''$ pulls back to
$e\cdot \mathcal{V}_0$ for some
integer $e\ge 0$.

We claim that $\mathcal{V}\times_{\Spec R}\cC\to X[2]^+$
extends to a map
$$\phi:\mathcal{V}\times_{\Spec R}\overline{\cC}\dashrightarrow X[2]^+$$
well-defined on an open subset containing the points
$(u,u',u)$ and $(u,u',u')$ for general $(u,u')\in\mathcal{V}_0$.
Given this, we then may define $\mathcal{U}$ by
deleting from $\mathcal{V}$ the points $(u,u')\in \mathcal{V}_0$ for which
$\phi$ is not defined at $(u,u',u)$ or at $(u,u',u')$, then the
desired conclusion holds since
the extension $\overline{\cC}''\to X[2]^+$ mentioned in the
statement of the lemma may be obtained by resolving the indeterminacy of
$\phi\circ((\sigma_1,\sigma_2)\times 1_{\overline{\cC}}):
\overline{\cC}\dashrightarrow X[2]^+$
and stabilizing.

To prove the claim,
we let $\mathcal{B}$ denote the image of $\mathcal{V}$ by the first projection
morphism to $\overline{\cC}$.
There is a unique morphism $\mathcal{B}\to X[m+1]$ through which the
composite morphism $\mathcal{V}\to X[m+2]''\to X[m+1]$ factors.
Then for $\mathcal{B}\to X[m+1]\times [\A^1/\mathbb{G}_m]$ on the second factor
given by the divisor $e\cdot \mathcal{B}_0$, if we let
$\mathcal{W}\to\mathcal{B}$ denote an FM space corresponding to it
under the isomorphism
$\X[m+2]''_{m+1}\cong X[m+1]\times [\A^1/\mathbb{G}_m]$, then there is a
morphism $X[m+2]^+|_{\mathcal{V}}\to\mathcal{W}$, fitting into a cartesian
diagram with $\mathcal{V}\to \mathcal{B}$, compatible with $m+1$ sections
and fixing $X$.
By stability
the restriction $X[m+2]^+|_{\Spec K\times_{\Spec R}\mathcal{V}}\to
\Spec K\times_{\Spec R}\mathcal{W}$ is compatible with the unique
morphisms to $X[m]^+|_{\Spec K}$ preserving $m$ sections and fixing $X$.
Forgetting the first $m$ sections, we obtain a level-2 FM space
$\mathcal{Z}$ over $\mathcal{B}$ and morphism $\mathcal{W}\to\mathcal{Z}$
compatible with $X[m+2]^+|_{\mathcal{V}}\to X[2]^+|_{\mathcal{V}}$.
Then we obtain
$\mathcal{B}\times_{\Spec R}\overline{\cC}\dashrightarrow \mathcal{Z}$
compatible with $\mathcal{V}\times_{\Spec R}\overline{\cC}\dashrightarrow
X[2]^+|_{\mathcal{V}}$ in the sense that if
$\mathcal{D}\subset \mathcal{B}\times_{\Spec R}\overline{\cC}$ is open
on which the map to $\mathcal{Z}$ is well-defined, then
$\mathcal{V}\times_{\Spec R}\overline{\cC}\dashrightarrow X[2]^+|_{\mathcal{V}}$
will be well-defined on $\mathcal{V}\times_{\mathcal{B}}\mathcal{D}$,
together fitting into a commutative diagram.
We may take $\mathcal{D}$ to contain the generic point of $E\times E$,
then we clearly have
$(u,u',u')\in \mathcal{V}\times_{\mathcal{B}}\mathcal{D}$ for general
$(u,u')\in \mathcal{V}_0$.
\end{proof}

\begin{Lemma}
\label{pullout}
Let $(f:\cC\to X[m]^+, p_1,\ldots,p_n)$ be a stable ramified map over $K$
such that the usual stable map compactification has a component $E$
of the closed fiber mapped into a relative singular locus of
$X[m]^+/X[m]$.
Then there exists a neighborhood $\mathcal{U}$ of the generic point of $E$,
such that for any section
$\sigma: \Spec R\to \mathcal{U}$, we have
$\cC\to X[m]^+|_{\Spec K}\cong X[m+1]^+|_{\Spec K}$ whose extension fits into
a commutative diagram
\[\xymatrix{
{\overline{\cC}'}\ar[r]\ar[d]_c & {X[m+1]^+}\ar[d] \\
{\overline{\cC}}\ar[r] & {X[m]^+}
}\]
of stable maps over $R$,
where $c$ is a contraction, such that the lifted section
$\sigma':\Spec R\to \overline{\cC}'$ meets the component
$E'\subset \overline{\cC}'_0$ corresponding to
$E\subset \overline{\cC}_0$.
\end{Lemma}

\begin{proof}
As in the proof of Lemma \ref{generalpairsections}, there exists
open $\mathcal{V}\subset \overline{\cC}$ and morphism
$\mathcal{V}\to X[m+1]$ lifting $\Spec R\to X[m]$,
such that the induced $\mathcal{V}\to \X[m+1]'_m$ factors up to 2-isomorphism
through $\Spec R$.
The last conclusion follows from a description
of $\X[m+1]'_m$ in a neighborhood of a point where the marked component
is ruled with no sections.
Taking $I\subset\{1,\ldots,m\}$ to be the set of sections on all the
descendants of the marked component, and
$g=0$ to be a defining equation for $\Delta_I\subset X[m]$ in a
neighborhood $U$ of the image point, then an open neighborhood in $\X[m+1]'_m$ may
be identified with the stack in which an object consists of morphism to $U$,
pair of line bundles each with regular section, and trivialization of the
tensor product of the line bundles taking the product of
the sections to the function $g$.
Then there is the extended FM space $\mathcal{W}\to \Spec R$,
and we may take $\mathcal{U}$
to be the complement in $\mathcal{V}$ of the indeterminacy locus of
$\overline{\cC}\dashrightarrow \mathcal{W}$.
\end{proof}

\subsection{Valuative criteria}
\label{valuativecrit}
We want to show that the stack $\overline{\mathfrak{U}}_{g, \mu }(X,\beta )$ of stable ramified maps is
proper over $\bk$.
We do this using the valuative criterion for properness stated in
\cite{DM}, by first verifying the existence of extensions as indicated,
then checking that the valuative criterion for separatedness is satisfied.

\begin{Prop}\label{proper}
Given an $n$-pointed stable ramified map $(f,{\bf p})$ over the
quotient field $K$ of a DVR $R$, there is a stable ramified map
extension $(\tilde{f},{\tilde{\bf p}} )$ of $(f,{\bf p} )$,
namely, $(\tilde{f},{\tilde{\bf p}})$ is a stable ramified map
defined over a DVR $R'$ which is the integral closure of $R$ in a
finite extension $K'$ of $K$ and
$(\tilde{f}|_{K'},{\tilde{\bf p}}) \cong (f,{\bf p} )|_{K'}$.
\end{Prop}

It is enough to prove the above when $R=\bk[[t]]$. Let
$R=\bk[[t]]$ from now on in this section.

\bigskip

\noindent{\em Proof of Proposition \ref{proper} when $\dim X=1$.}
Let $\cW /K$ be the pullback of the  universal family
$X[m]^+$ by a map $\Spec K\ra X[m]$. Then, using the image of the
marked points under $\cC\ra \cW\ra X[m]^+$, there is an induced
map $\Spec K\ra X[m:n]$. Thus, we may assume that $\cW$ is the
pullback of $X[n]^+$ under $g:\Spec K\ra X[n]$ such that $f\circ
p_i=\sigma _i\circ g$. Now take the usual stable map limit of this
$f: \cC\ra X[n]^+$.  The extension is a
stable ramified map as shown in the proof of Theorem 4 in
\cite{HM} using Abhyankar's Lemma (which removes any possible dimension
$1$ branch locus in the special fiber), the purity of the branch locus
(which removes any possible dimension $0$ branch locus),
and the universal covering of an $A_l$-singularity
(which shows the admissibility).
 \hfill$\Box$

\bigskip

\noindent{\em Local Analysis.}  Let $(\pi: \cC\ra \Spec R,\{
p_1,\ldots,p_n\}, f:\cC\ra X)$ be an {\em ordinary stable map} over
$\Spec R$ such that:

\begin{itemize}

\item Restricted to the generic fiber $\cC _K$, $f$  is a stable
$\mu$-ramified  map.

\item Restricted to the closed fiber $\cC _0$, $f$ is no longer
$\mu$-ramified.

\end{itemize}

 Here $\mathcal{W}=X\times \Spec R$. Assume that $\dim X=r\ge 2$ and
$\cC$ is {\em irreducible}. Then only the following cases are
possible unless some pair of markings have the same image over the
closed point of $\Spec R$.

\begin{enumerate}

\item (Jump of a Ramification Index) There exists a smooth point
$p$ on a component of $\cC _0$ non-contracted under $f$ such that
the ramification index jumps at $p$.

\item (Creation of a Nodal Point) There exists a nodal point $p$
of $\cC _0$ such that any component of $\cC _0$ containing $p$
does not contract under $f$.

\item (Contraction of a Component) There exists a contracted
component $E$ of $\cC _0$ and there exist no pair of ramification
markings approaching $E$.

\end{enumerate}

\begin{Lemma}\label{analysis}  In Case (1) or (2), the tangent line
map
$$\PP Tf: \cC\dashrightarrow \PP
TX$$ is not well-defined at $p$. In Case (3), $(\PP Tf)(E)$ is not
a point or there is a point $q\in E$ such that $\PP Tf$ is
not well-defined at $q$. \end{Lemma}

\begin{proof}
\emph{Case (1).}  Let
$m$ be the ramification index of $f|_{\cC _0}$ at $p$, and let
$x,t$ be local parameters of $\mathcal{O}_{p}$, where $t$ is
the local parameter of $R$.  Let $f=(f_1,\ldots,f_r)$
after introducing a local coordinate system of $X$ at $f(p)$. Without
loss of generality, we
assume that the tangent line direction of the image curve of
$x$-axis under $f(x,0)$ is $[1,0,\ldots,0] \in \PP ^{r-1}=\PP
T_{f(p)}X$. Consider the equation $\frac{\partial f_1}{\partial x}=0$.
It defines a curve $Z$ not containing the
$x$-axis, but containing the origin $(x,t)=(0,0)$. If the curve
has at least two irreducible components, then $f$ is generically unramified
along one of the components by assumption of stable ramified maps.
Hence along the component, $\PP Tf$ approaches a point different from
$[1,0,\ldots,0]$. If the curve $Z$ is irreducible, then along the
curve, $\PP Tf$ approaches a point different from $[1,0,\ldots,0]$ since the
generic point of the curve is a zero of multiplicity $m-1$ of
$\frac{\partial f_1}{\partial x}$, while the generic point
cannot be a zero of multiplicity $m-1$ or higher of the derivative
$\frac{\partial f_i}{\partial x}$ for all $i\ne 1$
(otherwise, along the curve, $f$ has ramification index at least $m$).

\bigskip

\emph{Case (2).} The domain surface $\cC$ is \'etale locally
defined by the relation $xy=t^k$ at $(0,0,0)$ for some positive
integer $k$, with $\pi (x,y,t)= t$. Furthermore, $f(x,y,t) \in R
[[x,y]]^{\oplus r}$, $f(x,0,0)\neq 0$, and $f(0,y,0)\neq 0$.
We may suppose that $f(0,0,t)=0$.
The argument is based on the set of lattice points $(\delta,\nu)$ of monomials
$x^\delta t^\nu$ ($\delta\ge -\nu/k$, $\nu\ge 0$) appearing in $f$
(Figure \ref{figuretwo}).

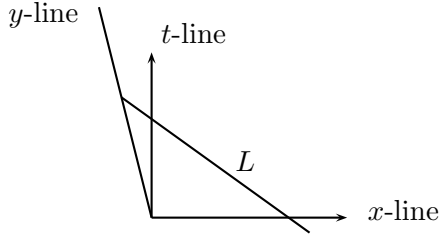
\begin{figure}
\begin{pspicture}(-2,-0.2)(4.8,2.9)
\psline{->}(0,0)(0,2.2)\psline{->}(0,0)(2.6,0)
\psline(-0.4,1.6)(2.1,-0.2)
\psline(0,0)(-0.7,2.8)\uput[135](1.5,0.5){$L$}
\uput[0](2.7,0.1){$x$-line} \uput[45](0,2.2){$t$-line}
\uput[180](-0.8,2.7){$y$-line}
\end{pspicture}
\caption{$(\delta ,\nu )$-plane}
\label{figuretwo}
\end{figure}

There is a line $L$ containing two distinct exponents
$x^{\alpha_i}t^{\beta_i}$, $i=1$, $2$ of $f$ with
\[
\alpha_1>0,\qquad
\beta_1\ge 0,\qquad
\alpha_2<0,\qquad
\beta_2\ge -k\alpha_2,
\]
and with all exponents of $f$ contained in the closed half-plane bounded
below by $L$.
We define $\ell=-(\beta_1-\beta_2)/(\alpha_1-\alpha_2)$.
Notice that $\ell<k$, since the half-plane contains a lattice point
corresponding to a power of $y$.
If, among all exponents $x^\alpha t^\beta$, $(\alpha,\beta)\in L$, there is
some pair for which the coefficient vectors of $f$ are linearly
independent, then $\PP Tf(x,y,t)$ approaches different points along paths
\[ (x,y) = (ct^\ell, c^{-1}t^{k-\ell}) \]
as $c\in \bk^\times$ varies.
In this case, $\PP Tf(p)$ is not well-defined.

Otherwise there is a constant tangent direction along such paths,
which without loss of generality we take to be
$[1,0,\ldots,0]\in \PP^{r-1}=\PP T_{f(p)}X$.
Putting
\[x=zt^\ell, \qquad y=z^{-1}t^{k-\ell},\]
we have
\[ f(zt^\ell, z^{-1}t^{k-\ell},t)=t^{\frac{\alpha_1\beta_2-\alpha_2\beta_1}{\alpha_1-\alpha_2}}g(z,t) \]
for some $g=(g_1,\ldots,g_r)\in \bk[z,z^{-1}][[t]]^{\oplus r}$ with
$g_i(z,0)=0$ for $i\ge 2$, while $g_1(z,0)$ is a Laurent polynomial
containing at least one positive and one negative power of $z$.
It follows that there exists $c\in\bk^\times$
such that $\frac{\partial g_1}{\partial z}$ vanishes
at $(c,0)$.
We are in the situation of Case (1):
$f$ is generically unramified, and $\PP Tf$ approaches
a point different from $[1,0,\ldots,0]$, on any
irreducible component of the curve defined by
$\frac{\partial g_1}{\partial z}=0$.

\emph{Case (3).} Set $p=f(E)$.
We suppose that $(\PP Tf)(E)$ is the point
$[1,0,\ldots,0]\in \PP^{r-1}=\PP T_pX$.
We may choose a projection to $\PP^1$ from $X$,
defined in a neighborhood of $p$, so that with $p'$ the image of $p$ in $\PP^1$
the image in $T_{p'}\PP^1$ of $(1,0,\ldots,0)\in T_pX$ is different from zero,
the composite $f':\cC\to\PP^1$ is
well-defined, stable, and over $\Spec K$ is a stable ramified map,
and the only irreducible components of $\cC_0$ mapping by $f'$ to $p'$
are those mapping to $p$ by $f$.
Now we consider the stable ramified limit
$\cC'\to \cW$ over $\Spec R$ of $\cC'\to \PP^1$ over $\Spec K$.
There exists a component $E'\subset \cC'_0$ over
$E\subset \cC_0$
(i.e.\ mapping into $E$ via stabilization $\cC'\to\cC$)
containing a smooth point that is a ramification point of
$\cC'_0\to \cW_0$ and is not the
limit of any of the $p_i$.
A pair of general sections approaching $E'$
gives rise to a new target space $X[2]^+$ and a new stable map
limit $\widetilde{\cC}\to X[2]^+$.

The hypothesis concerning $(\PP Tf)(E)$ implies that
the image of every irreducible component of $\widetilde{\cC}_0$ over $E$
meeting $(X[2]^+)^{\mathrm{sm}}$ is a line intersecting
$(X[2]^+)^{\mathrm{sing}}=\PP^{r-1}$ at $[1,0,\ldots,0]$.
Now, considering the new target space $\PP^1[2]^+$ with its stable map
limit $\widetilde{\cC}'$, on which there is the component
$\widetilde{E}'$ corresponding to $E'\subset \cC'_0$,
we see using Lemma \ref{generalpairsections} that there is a corresponding component
$\widetilde{E}\subset \widetilde{\cC}_0$, over $E$, having image a line
in $X[2]^+$ and a ramification point which is not the limit of any of the $p_i$.
We conclude by the argument of Case (1).
\end{proof}

\bigskip

\begin{proof}[Proof of Proposition \ref{proper} when $\dim X\ge 2$.]
We assume that the target $\mathcal{W}$ is $X[m]^+|_{\Spec K}$ for some
given $\Spec K\to X[m]$.
By Lemma \ref{TLMC} we have the tangent line map
$\PP(Tf):\cC\to \PP(TX)$ and hence its ordinary stabilization with
the $n$ markings
\[ \PP(Tf)':\cC'\to \PP(TX) \]
where $\cC'$ is a suitable contraction of $\cC$.
With stable map extension
$\overline{\cC}'\to \PP(TX)$ (with $n$ markings)
we apply Lemma \ref{generalpairsections} to each component $E_i$ of
$\overline{\cC}'_0$ to get open sets
$\mathcal{U}^{(i)}\subset \overline{\cC}'\times \overline{\cC}'$ where
we may suppose $\mathcal{U}^{(i)}_0\subset
(E_i^{\mathrm{sm}}\times E_i^{\mathrm{sm}})\setminus
\Delta_{E_i^{\mathrm{sm}}}$ for $i=1$, $\ldots$, $c$ where $c$ denotes the
number of components of $\overline{\cC}'_0$.
There is an open subset $\mathcal{U}$ of the product of the
$\mathcal{U}^{(i)}$, with $\mathcal{U}_0\ne\emptyset$, on which we may
apply Proposition \ref{PairOper} to obtain
$\Spec K\to X[m,2c]$ with $X[m,2c]^+|_{\Spec K}\cong \mathcal{W}$.

Arguing as in the proof of Lemma \ref{generalpairsections}, we see that
for a general section $\Spec R\to \mathcal{U}$ the
stable map extension
$\overline{\cC}''\to X[2c]^+\times \PP(TX)$ of the
map $\cC\to X[2c]^+ \times \PP(TX)$
(the composite $\cC\to X[m,2c]^+\to X[2c]^+$ on the first factor,
and the map $\PP(Tf)$ on the second factor)
has the property that the composite
$\overline{\cC}''\to X[2c]^+\times \PP(TX)\to X[2c]^+$
remains stable.
Hence the same holds if we extend using $n+2c$ sections; we therefore
obtain the stable map extension $\overline{\cC}''\to X[n+2c]^+$.

Over $K$ we have $X[m,n+2c]^+\to X[n+2c]^+$ contracting only ruled components;
correspondingly on $\cC\to \cC'$ each contracted component has a line
non-fiber image
of some end rational component of $\cC$, joined by a chain of
$\PP^1$'s to a component that survives in $\cC'$.
Let $j$ denote the total number of collapsed ruled components.
We may put back these collapsed components by adding $j$ sections.
Indeed, if we let $\mathcal{I}$ denote the set of subsets of
$N\sqcup \{1,\ldots,2c\}$ describing the image of $\Spec K\to X[n+2c]$,
then the target $X[m,n+2c]^+|_{\Spec K}$ is contained in a closed substack
of $\X[n+2c+j]_{n+2c}$ isomorphic to
\[\bigg(\bigcap_{I\in\mathcal{I}}\Delta_I\bigg)\times (B\mathbb{G}_m)^j,\]
where $B\mathbb{G}_m$ denotes the classifying stack of $\mathbb{G}_m$.
Hence we have $\X[n+2c+j]_{n+2c}^+|_{\Spec K}\cong\mathcal{W}$, determined up
to an element of $(\mathbf{k}((t))^\times)^j$.
The isomorphism class over $\Spec R$ is therefore determined by
$j$ integers.
These are uniquely determined by the requirement that
each line image of a rational end component mentioned above should in
the limit tend to a non-fiber image, with at least one line per
component tending in the limit to a non-fiber image not contained in the
relative singular locus.
For such an $R$-point of $\X[n+2c+j]_{n+2c}^+$ we consider a lift
over $\Spec K$ to $X[n+2c+j]$ which factors through $X[m:n+2c+j]$
(compatibly with $\Spec K\to X[m]$)
and use this to determine a stable map limit
$\overline{\cC}'''\to X[n+2c+j]^+$ over $R$.

Using Lemmas \ref{pullout} and
\ref{analysis} it follows from the stability condition on
$\overline{\cC}'''\to X[n+2c+j]^+$ that in the special fiber there
are no contracted components.
Some finite number of components may be mapped to lines in relative singular
loci of $X[n+2c+j]^+/X[n+2c+j]$.
If the number of such components is $k$ then
we obtain by repeated use of Lemma \ref{pullout} a
target $\widetilde{\mathcal{W}}:=X[n+2c+j+k]^+|_{\Spec R}$
with isomorphism
$\widetilde{\mathcal{W}}|_{\Spec K}\cong \mathcal{W}$, such that the stable map limit
$\widetilde{C}\to \widetilde{\mathcal{W}}$ fulfills the conditions of the
proposition.
Indeed, AC is automatically satisfied by Proposition 2.2 of \cite{J.Li};
PRIC holds by Lemma \ref{analysis};
we have DPC because the $n$ sections are among those defining the
target $\widetilde{\mathcal{W}}$;
and SC holds by construction, using Lemma \ref{analysis}.
\end{proof}

Finally we come to the main result of the paper.

\begin{Thm}\label{MainThm} The stack $\overline{\mathfrak{U}}_{g,\mu }(X,\beta)$
of  $(g, \beta,\mu)$-stable ramified maps
to FM degeneration spaces of
$X$ is a proper
Deligne-Mumford stack over $\bk$.
\end{Thm}
\begin{proof}
It remains only to verify the valuative criterion for separatedness;
then properness follows from Proposition \ref{proper}.
Assume that for $i=1$, $2$ we have stable ramified maps
$f_i:\cC_i\to \mathcal{W}_i$,
with $\mathcal{W}_i=g_i^*X[n_i]^+$ for some
$g_i:\Spec R\to X[n_i]$, and an isomorphism of stable ramified maps over $K$,
i.e., pair of isomorphisms $(\varphi,\psi)$
\[\varphi:\cC_1|_{\Spec K}\to \cC_2|_{\Spec K},
\qquad
\psi:\mathcal{W}_1|_{\Spec K}\to \mathcal{W}_2|_{\Spec K}\]
satisfying
\[\psi\circ f_1|_{\Spec K}=f_2|_{\Spec K}\circ\varphi\]
where $\psi$ fixes $X$ and $\varphi$ preserves the $n$ sections of the
$\cC_i$.
By the uniqueness of the extension of ordinary stable maps, it suffices to
verify that targets $\mathcal{W}_1$ and $\mathcal{W}_2$ are isomorphic by
an extension of $\psi$. To do so, we use the Tangent Line Map Condition.

By Lemma \ref{TLMC}, we have the tangent line map
$$\PP (Tf_i) : \cC _i \rightarrow \PP (TX)$$
and hence its ordinary stabilization with the $n$ markings
$$ (\PP(Tf_i))' : \cC _i' \ra \PP (TX)$$
where $\cC _i'$ is a suitable contraction of $\cC _i$. Since the
$n$-pointed stable maps $(\PP Tf_1)'$ and $(\PP Tf_2)'$ are
equivalent over $K$, they are equivalent over $R$. Hence there is
an isomorphism $\varphi ': \cC _1' \ra \cC _2'$ satisfying that:
\begin{itemize}
\item the diagram
 \[\xymatrix{
              \cC_1|_K \ar[r]^\varphi \ar[d] & \cC_2|_K \ar[d]\\
              \cC_1'|_K \ar[r]^{\varphi'|_K} & \cC_2'|_K
 }\] commutes;

 \item $(\PP Tf_1)'= (\PP Tf_2)'\circ\varphi '$; and

 \item the $n$ markings are preserved under $\varphi '$.

 \end{itemize}

We consider all components
of the closed fiber of $\cC'_1$;
for each such component, there is the corresponding one in the closed fiber of
$\cC _1$. Consider two general points on each such component and
two sections passing through those points. These sections together
with the $n$ markings form $n+2c$ sections of $\cC _1\ra \Spec R$,
where $c$ is the number of components of the closed fiber of $\cC'_1$. By
$\varphi '$, we obtain the corresponding sections of $\cC _2 \ra
\Spec R$.  Using DPC, we may assume that
the images of those $n+2c$ sections under $f_i$ are pairwise distinct.
Now
applying Proposition \ref{PairOper} to $g_i$ with those $n+2c$ sections, we
obtain an isomorphism of contractions of
$\mathcal{W}_1$ and
$\mathcal{W}_2$.

Let $j$ denote the number of contracted components of
$\mathcal{W}_i\to X[n+2c]^+$ over $K$,
and let us add $j$ new sections as described in the proof of
Proposition \ref{proper}.
Notice that the choice of expanded target, up to isomorphism, is determined
completely by restriction over $K$ of the given $(\cC_i,\mathcal{W}_i)$,
hence we obtain an isomorphism of contractions of $\mathcal{W}_i$
restricting to an isomorphism over $K$.

By the uniqueness of stable map extension, the number of components
of the special fiber of $\cC_i$ landing in the relative singular locus
of $X[n+2c+j]^+$ is the same for $i=1$, $2$; let us call it $k$ and
now add $k$ new sections of $\cC_1\to \Spec R$ and the corresponding ones of
$\cC_2\to \Spec R$.
We claim that the induced maps
$\Spec R\to X[n_i,n+2c+j+k]$ are in fact in
$X[n_i:n+2c+j+k]$.
Stability is clear for any screen containing a non-line image,
since it will then have two out of the $2c$ markings.
It remains only to consider ruled screens containing none of the $n$ markings
and only line images.
Then there is some non-fiber line image.
If it is a limit of non-fiber line images on a ruled component over $K$
then it will be stable by one of the $j$ sections.
Otherwise the non-fiber line maps to a line in a singular component of
$X[n+2c+j]^+$, hence there will be one of the $k$ sections.
This shows that $\mathcal{W}_1$ and $\mathcal{W}_2$ are isomorphic under
an extension of $\psi$.
\end{proof}


\section{Ramified Gromov-Witten Invariants}\label{Deformation}
\subsection{Obstruction theory}
The approach to relative obstruction theory suggested by J. Li at the
beginning of \S1.2 of \cite{J.Li2} can be worked out in the case of the
moduli stack of stable ramified maps using Olsson's
deformation theory of log schemes \cite{Olsson3}.
If $(\mathcal{C}\ra S, \mathcal{W}\ra S, f : \mathcal{C}\ra \mathcal{W})$ is
a family of ramified stable maps then we have natural log
structures $M^{\mathcal{C}/S}$ on $\mathcal{C}$,
$M^{\mathcal{W}/S}$ on $\mathcal{W}$ and $N^{\mathcal{C}/S}$ and $N^{\mathcal{W}/S}$
on $S$ making $(\mathcal{C},M^{\mathcal{C}/S})\ra (S,N^{\mathcal{C}/S})$
and $(\mathcal{W},M^{\mathcal{W}/S})\ra (S,N^{\mathcal{W}/S})$
log smooth morphisms.
Following \S3B of \cite{Mo} there is a canonical log structure $N$ on $S$,
associated to the monoid pushout
$N^{\mathcal{C}/S}\oplus_{N'}N^{\mathcal{W}/S}$ where $N'$ is the
submonoid of $N^{\mathcal{C}/S}\oplus N^{\mathcal{W}/S}$ generated by
$(m\cdot \log(s'),\log(s))$ for every node of the geometric fibers of
$\mathcal{C}\ra S$, and if we let $(\mathcal{C},M)$ denote the log scheme
obtained as the fiber product
$(\mathcal{C},M^{\mathcal{C}/S})\times_{(S,N^{\mathcal{C}/S})}(S,N)$
then there is, canonically, $f^*M^{\mathcal{W}/S}\ra M$ making
\[\xymatrix{
{(\mathcal{C},M)} \ar[r] \ar[d] & {(\mathcal{W},M^{\mathcal{W}/S})} \ar[d] \\
{(S,N)} \ar[r] & {(S,N^{\mathcal{W}/S})}
}\]
a commutative diagram of fine log schemes \cite{Kato}.

\begin{Prop}\label{defo}
Let $S=\Spec A$ and let $I$ be an $A$-module.
Consider a square zero extension $B$ of $A$ by $I$.
Let $f : \mathcal{C} \ra \mathcal{W} $ be a stable ramified map over $S$.
Let $(\tilde{\mathcal{C}}\ra \tilde S,\{\tilde p_1, \ldots,
\tilde p_n \})$ and $(\tilde{\mathcal{W}}\ra \tilde S, \tilde{\mathcal{W}}\ra X)$
be extensions of $(\mathcal{C}\ra S, \{p_1, \ldots, p_n\})$ and
$(\mathcal{W}\ra S, \mathcal{W}\ra X)$ over $\tilde{S}=\Spec B$,
let $\tilde q_i : S \ra \mathcal{W}$ be extensions of
$q_i := f\circ p_i$ for $i=1$, $\ldots$, $n$, and
let $\tilde{N}$ be a fine log structure on $\tilde{S}$ extending the
log structure $N$ on $S$, together with morphisms
$N^{\tilde{\mathcal{C}}/\tilde{S}}\ra \tilde{N}$ and
$N^{\tilde{\mathcal{W}}/\tilde{S}}\ra \tilde{N}$ extending the ones over $S$.
Then there is a natural element
\[ob(f,I)\in H^1(\mathcal{C},f^*T^\dagger_{\mathcal{W}}(-\mu_1 p_1-\cdots-\mu_n p_n)
\otimes_{\mathcal{O}_S}I)\]
of obstruction to extension to a stable ramified map
$\tilde f : \tilde{\mathcal{C}} \ra \tilde{\mathcal{W}}$ over
$\tilde S$ such that $\tilde N'$ and
$\tilde N'\ra N^{\tilde{\mathcal{W}}/\tilde{S}}$ are compatible with the given
$N^{\tilde{\mathcal{C}}/\tilde{S}}\ra \tilde{N}$ and
$N^{\tilde{\mathcal{W}}/\tilde{S}}\ra \tilde{N}$.
When the obstruction vanishes, the extensions $\tilde f$ satisfying the
compatibility are a torsor under
\[H^0(\mathcal{C},f^*T^\dagger_{\mathcal{W}}(-\mu_1 p_1-\cdots-\mu_n p_n)
\otimes_{\mathcal{O}_S}I).\]
\end{Prop}

\begin{proof}
We use Theorem 5.9 in \cite{Olsson3}.
To enforce PRIC, we replace $\mathcal{W}$ with the (non-separated) union
\[\mathcal{W}\cup_{\mathcal{W}\setminus\{q_1,\ldots,q_n\}}
\mathrm{Bl}_{\{q_1,\ldots,q_n\}}\mathcal{W}\]
where the log structure is the standard one on $\mathcal{W}$ and
the standard one plus the exceptional divisor on $\mathrm{Bl}_{\{q_1,\ldots,q_n\}}\mathcal{W}$.
Then $f$ gets replaced by a map $f'$ sending
$\mathcal{C}\setminus \{p_1,\ldots,p_n\}$ to $\mathcal{W}$ and by the
lift to $\mathrm{Bl}_{\{q_1,\ldots,q_n\}}\mathcal{W}$
in a neighborhood of $p_i$ (which is well-defined by Remark \ref{RmkRam}),
and near $p_i$ we have an identification of log tangent sheaves
$f'^*T^\dagger_{\mathrm{Bl}_{\{q_1,\ldots,q_n\}}\mathcal{W}}\cong
f^*T^\dagger_{\mathcal{W}}(-\mu_i p_i)$.
\end{proof}

By standard machinery we have a perfect obstruction theory
(\cite{BF})
\[ R \pi_*(f^*T^\dagger_{\mathcal{W}}(-\mu_1 p_1-\cdots-\mu_n p_n))^\vee \ra
L^\bullet_{\overline{\mathfrak{U}}_{g, \mu }(X,\beta )/\mathfrak{B}} \]
relative over the base stack $\mathfrak{B}$ of
curves (prestable $n$-pointed of genus $g$), FM spaces of $X$ with
$n$-tuples of smooth pairwise distinct points,
fine log structures, and pairs of morphisms of log structures
\[(\mathcal{C}\ra S, \mathcal{W}\ra S, N,
N^{\mathcal{C}/S}\ra N,
N^{\mathcal{W}/S}\ra N).\]
By Theorems 1.1 and 4.6 and Corollary 5.25 of \cite{Olsson1}
and Proposition 2.11 of \cite{Olsson3}
the stack $\mathfrak{B}$ is algebraic and pure-dimensional, so by
\cite{BF} there is a virtual fundamental class
$[\overline{\mathfrak{U}}_{g,\mu }(X,\beta )]^{\mathrm{vir}}$.
Alternative ways to define a virtual fundamental class are
by adding auxiliary log structures (\cite{Kim}, see also \cite{Lopez})
or orbifold structures (\cite{AF}).

Using the natural morphism
\[ Te_i: \overline{\mathfrak{U}}_{g,\mu}(X,\beta ) \ra \PP TX ,\]
we define \emph{ramified Gromov-Witten invariants} by
\[ \int _{[\overline{\mathfrak{U}}_{g,\mu}(X,\beta
)]^{\mathrm{vir}}} \prod \psi _i^{a_i} Te_i^* (\gamma _i) ,\] where
$\psi _i$ are gravitational descendants associated to the $i$-th
marking and $\gamma _i$ are cohomology classes of $\PP TX$.
These invariants are deformation invariant just like the usual Gromov-Witten invariants. 
When $\mu=(1,\ldots,1)=(1^n)$ we write
$\overline{\mathfrak{U}}_{g,n}(X,\beta)$ for
$\overline{\mathfrak{U}}_{g,\mu}(X,\beta)$ and speak of
\emph{unramified Gromov-Witten invariants}.

\subsection{Pandharipande's conjectures}
\label{rahul} In this subsection let ${\bf k}$ be the field of complex numbers.
For threefold targets,
R. Pandharipande has proposed a conjectural link between
unramified Gromov-Witten invariants and certain integer quantities that
can be defined with usual Gromov-Witten theory.
A consequence would be the following statement.

\begin{Conj}
\label{conj1}
Let $X$ be a smooth projective three-dimensional algebraic variety
and $\beta$ a curve class on $X$ with $\int_\beta c_1(T_X)>0$.
Then the unramified Gromov-Witten invariants on $X$ defined by
integral cohomology classes on $X$, without gravitational descendants, are
integers.
\end{Conj}

Let us call a curve class \emph{locally Fano} if it has positive intersection
number with the anticanonical class.
Generalizing the BPS counts from string theory,
Pandharipande defined numbers
$n_{g,\beta}(\gamma_1,\ldots,\gamma_n)$ which for locally Fano $\beta$ are
given by
\[
\sum_{g\geq 0}
n_{g,\beta}(\gamma_1,\ldots,\gamma_n) \lambda^{2g-2}
\left(\frac{\sin(\lambda/2)}{\lambda/2}\right)^{2g-2+\int_\beta c_1(T_X)}
=\sum_{g\geq 0}
\langle \gamma_1,\ldots,\gamma_n \rangle_{g,\beta} \lambda^{2g-2}
\]
where the $\langle \gamma_1,\ldots,\gamma_n \rangle_{g,\beta}$ are usual
Gromov-Witten invariants and $\lambda$ is a formal variable (see \cite{Pandharipande1, Pandharipande2}).
He conjectured that the $n_{g,\beta}(\gamma_1,\ldots,\gamma_n)$ are
integers, provided that the $\gamma_i$ are integral cohomology classes.
This was proved in the locally Fano case by Zinger \cite{Zinger}.

According to Pandharipande,
the unramified Gromov-Witten invariants in locally Fano curve classes
of cohomology classes on $X$
should yield BPS counts directly.
Notice that by Zinger's result, this statement would imply Conjecture \ref{conj1}.

\begin{Conj}
\label{conj2}
If $\beta$ is a locally Fano curve class on a nonsingular projective
threefold $X$, then
for any cohomology classes $\gamma_1$, $\ldots$, $\gamma_n$ on $X$,
\[ \int _{[\overline{\mathfrak{U}}_{g,n}(X,\beta
)]^{\mathrm{vir}}} \prod e_i^* (\gamma _i) =
n_{g,\beta}(\gamma_1,\ldots,\gamma_n).\]
\end{Conj}

\end{document}